\documentclass{amsart}
\usepackage{latexsym}
\usepackage{amsmath,amsfonts,epsfig, graphics, graphicx, amsthm, amssymb}
\input{xy}
\xyoption{all}

\newtheorem{theorem}{Theorem}[section]
\newtheorem{lemma}[theorem]{Lemma}

\newtheorem{varthmC}{Conjecture}
\newtheorem{varthm1}{Main Theorem}

\newtheorem*{Ko}{Korselt's Criterion}
\newtheorem*{Mo}{Motivating Question}

\begin{document}
\title{Variants of Korselt's Criterion}
\author[T. Wright]{Thomas Wright}
\maketitle

\begin{abstract}
Under sufficiently strong assumptions about the first term in an arithmetic progression, we prove that for any integer $a$, there are infinitely many $n\in \mathbb N$ such that for each prime factor $p|n$, we have $p-a|n-a$.  This can be seen as a generalization of Carmichael numbers, which are integers $n$ such that $p-1|n-1$ for every $p|n$.
\end{abstract}

\section{Introduction}
Recall that a Carmichael number is a composite number $n$ for which

$$a^n\equiv a\pmod n$$
for every $a\in \mathbb Z$.  As Fermat's Little Theorem states that the above congruence is true whenever $n$ is prime, Carmichael numbers thus serve as the disproof of the converse of Fermat's Little Theorem.


While the first Carmichael numbers were discovered in 1910 by R.D. Carmichael \cite{Ca}, the search for Carmichael numbers was aided by the discovery of a necessary and sufficient condition by A. Korselt \cite{Ko} in 1899:

\begin{Ko} \textit{$n$ is a Carmichael number if and only if $n$ is squarefree and $p-1|n-1$ for each prime $p|n$.}\end{Ko}

Much like the relationship between Carmichael numbers and Fermat's Little Theorem above, Korselt's criterion is also seen to be easily satisfied if $n$ is prime.

Korselt's criterion is of fundamental importance in the study of Carmichael numbers, as it remains the primary tool used to prove statements about such numbers.  Because of its importance in the study of such pseudoprimes, mathematicians have begun to ask questions about what might happen should one wish to generalize Korselt's criterion.  The most obvious of these generalizations is to change the 1 to another number; it is this generalization which motivates the present paper:

\begin{Mo} For any $a\neq 1$, are there infinitely many composite $n$ such that for every prime $p|n$, $p-a|n-a$?
\end{Mo}

The question is not new.  In their 1994 proof the infinitude of Carmichael numbers, Alford, Granville, and Pomerance \cite{AGP} stated the following:\\

\textit{One can modify our proof to show that for any fixed nonzero integer $a$, there are infinitely many squarefree, composite integers $n$ such that $p-a$ divides $n-1$ for all primes $p$ dividing $n$.  However, we have been unable to prove this for $p-a$ dividing $n-b$, for $b$ other than 0 or 1.  Such questions have significance for variants of pseudoprime tests, such as the Lucas probable prime test, strong Fibonacci pseudoprimes, and elliptic pseudoprimes.}\\

Our question, then, is the specific case of $a=b$.  For purposes of nomenclature, we will refer to a number $n$ for which $p|n$ implies $p-a|n-a$ as an \textit{$a$-Carmichael number}.  A regular Carmichael number is then a 1-Carmichael number.



There has been little progress on this problem since 1994.  In fact, there was no progress at all until 2011, when Ekstrom, Pomerance, and Thakur \cite{EPT} gave a conditional proof of infinitude in the case of $a=b=-1$; this result was later proven unconditionally in a 2013 paper by the present author \cite{WrE}.  As of now, however, -1 remains the only value for $a$ (besides 1 and 0) where anything has been proven, even conditionally.

In this paper, we use a conjecture of Heath-Brown to resolve the case of $p-a|n-a$ for every $a\in\mathbb Z$.  As this divisibility condition is easily satisfied when $n$ is prime, our theorem can be seen as a reasonable generalization in the search for pseudoprimes.  It is not clear that our results can be pushed to cases of $a\neq b$; this is in part because it is not clear for which cases of $a$ and $b$ one expects infinitely many $n$.  In the next section, we will discuss the conjecture of Heath-Brown and our results.





\section{Conjectures about primes in arithmetic progressions}

The results of this paper will hinge upon the size of the first prime in an arithmetic progression.  The standing conjecture in the area was made by Roger Heath-Brown \cite{HB}, who claimed the following:

\begin{varthmC} Let $(c,m)=1$.  Then the smallest prime $p$ that is congruent to $c$ mod $m$ is $\ll m(\log m)^2$.
\end{varthmC}

Other versions of this conjecture have previously been used to prove results about Carmichael numbers and associated constructs.  For instance, Banks and Pomerance \cite{BP} used a variant of this conjecture to prove infinitely many Carmichael numbers in arithmetic progressions, and Ekstrom, Pomerance, and Thakur \cite{EPT} used another form of this conjecture to prove the aforementioned result about numbers $n$ for which $p|n$ implies $p+1|n+1$.  In both cases, the authors actually used slightly weaker variants of the Heath-Brown conjecture; in fact, both papers showed that their results can be proven using either $m^{1+(\log m)^{\kappa-1}}$ for some $\kappa<1$ or $m^{1+\frac{\eta}{\log\log m}}$ for some $\eta>0$ as the upper bound.  While the results in \cite{BP} and \cite{EPT} have been majorized in subsequent papers by Matom\"{a}ki and the present author (see \cite{Ma}, \cite{WrC}, and \cite{WrE}), the idea for using Heath-Brown's conjecture to prove results about Carmichael numbers can be seen to come from these works.

In this paper, we stick primarily with the original Heath-Brown version of the conjecture, as it is not clear that our results would still hold with the variants mentioned above.  We do afford ourselves the following relaxation of the conjecture:


\begin{varthmC} \label{strong} Let $(c,m)=1$.  Then there exists some constant $A$ such that smallest prime $p$ that is congruent to $c$ mod $m$ is $\ll m(\log m)^A$.
\end{varthmC}

Our results would still hold if one assumed that the first prime in an arithmetic progression were $\ll m(\log m)^{\log \log m}$; however, we use the conjecture above for the purposes of transparency.


It should be noted that all of these conjectures about first primes in an arithmetic progression appear to be well beyond scope of modern mathematics.  Currently, the best bound for first prime in an arithmetic progression mod $c$ is $\ll c^{5.2}$, which was proven by Xylouris \cite{Xy} in 2009.  Even the assumption of the Generalized Riemann Hypothesis would only reduce this bound to $\ll c^{2+\epsilon}$.


Regardless, if we are afforded the conjectures above, we are able to resolve the problem of $a$-Carmichael numbers completely.


\begin{varthm1}
Assume Conjecture \ref{strong}.  Then for any $a$, there are infinitely many positive integers $n$ such that $p-a|n-a$ for each prime $p|n$.  In fact, there exists a constant $C$ such that the number of $a$-Carmichael numbers up to $X$ is
$$\gg X^{\frac{C}{(\log\log\log X)^2}}.$$
\end{varthm1}


The proof of the main theorem is a combination of the ideas of \cite{AGP}, \cite{EPT}, and \cite{WrF}, as well as some new ideas.  As in \cite{AGP}, we begin with an integer $L$ for which the maximum order of an element mod $L$ is small relative to $L$.  Using this, we use the conjecture to find primes of the form $dk+a$ with relatively small $k$ for many of the $d|L$.  From here, we establish that there exists a relatively small $k$ for which many of the $dk+a$ are prime simultaneously; then, we invoke the conjecture again to find that there exists a relatively small prime $P$ that is congruent to $a$ mod $kL$.  With this setup, we prove that there exists some collection of these primes $dk+a$ whose product is congruent to 1 mod $P-a$; multiplying this product by $P$ then gives us an $a$-Carmichael number.  Since there are infinitely many such $L$, there are infinitely many $a$-Carmichael numbers.

It is worth noting that there are several new ideas in this paper:\\

- In contrast to \cite{BP} and \cite{EPT}, we use the Heath-Brown Conjecture not once but twice; we first use it to find a suitable set of primes with which to work, and then we use it again to find another prime to append to this set in order to create an $a$-Carmichael number.

- Additionally, instead of simply taking $d$ to be any divisor of $L$, we first group the prime divisors of $L$ into sets of size $A+1$; we then take the $d$'s to be only those divisors of $L$ that can be written as products of some of these sets of $A+1$ primes.  Doing this ensures that we do not double-count any of the primes of the form $dk+a$.

- Most importantly, for our various choices of $d$, we do not attempt to find a density of primes of the form $dk+a$.  Instead, we only look to find a single prime of this form for each $d$.  This is the key to our approach because it dramatically lessens our value for $k$; Heath-Brown's conjecture can be applied for $k$ as small as $\log^A d$, while density estimates can only be used if $k$ is at least $d^\epsilon$ (even under Montgomery's conjecture for primes in arithmetic progressions [MV, Conjecture 13.9]).  Because $k$ is small, the order of $a$ mod $k$ will be small as well.






\section{Finding Primes: Conjecture Application \#1}
As in other papers on Carmichael numbers, we begin by letting $\mathcal Q$ be a collection of primes $q$ for which $q-1$ is relatively smooth.  More specifically, let $P(n)$ denote the largest prime factor of $n$, and let $1<\theta<2$.  Then we define

\[\mathcal Q=\{q\mbox{ }prime:\frac{y^\theta}{\log y}\leq q\leq y^{\theta},\mbox{ }q\equiv -1 \pmod{\alpha},\mbox{ }P(q-1)\leq y\}.\]

It is known that there exist $Y$, $\gamma$ such that if $y>Y$ then $|\mathcal Q|\geq \gamma \frac{y^\theta}{\log y^\theta}$ (see [Wr2, Lemma 2.1]).  From this, we construct $L$ to be the following:

$$L=\prod_{q\in \mathcal Q}q.$$

The usual procedure in these cases is to examine the number of primes of the form $dk+a$ for the various $d|L$.  From here, we then show that there are enough primes of this form to guarantee that there exists some $k$ for which there are many primes of this form.  Here, though, we make an alteration to ensure that we are not counting the same primes multiple times.  To this end, we define the following:

Let us index the divisors of $L$ as $q_1,q_2,....,q_{\omega(L)}$.  We will assume that Conjecture 2 is true for $\log^A d$ for some power $A$.  We then define

$$Q_i=q_{\scriptscriptstyle{(A+1)(i-1)+1}}q_{\scriptscriptstyle{(A+1)(i-1)+2}}...q_{\scriptscriptstyle{(A+1)i}}.$$

Note that for any such $Q_i$, we have

$$Q_i> \frac{y^{(A+1)\theta}}{\log^{A+1} y}.$$

We will only then consider divisors $d$ of $L$ such that $d$ can be written as the product of $Q$'s.  For a given $k$, let

$$\mathcal P_k=\{p=dk+a:p\mbox{ }prime,d|L,d=\prod_{j\in S}Q_j\mbox{ }for\mbox{ }some\mbox{ }S\subset \{1,2,...,\left[\frac{\omega(L)}{A+1}\right]\}\}.$$

\section{Goals and Bounds}
Our goal will be to show that there exists a $\mathcal P_k$ that is sufficiently large for our purposes.  Here, sufficiently large will be defined in terms of the following theorem of van Emde Boas and Kruyswijk \cite{EK} and Meshulam \cite{Me}:

\begin{theorem}\label{main}
Let $n(L)$ denote the smallest number such that any set of at least $n(L)$ elements of $(\mathbb Z/L\mathbb Z)^\times$ must contain some subset whose product of its elements is 1 mod $L$.  Let $\lambda(L)$ denote the maximal order (with regard to multiplication) of an element mod $L$.  Then
\begin{gather}
n(L)<\lambda(L)\left(1+\log\left(\frac{L}{\lambda(L)}\right)\right).
\end{gather}

Moreover, let $r>t>n(L)$.  Then any set of $r$ elements of $(\mathbb Z/L\mathbb Z)^\times$ contains at least $\left(\begin{array}{c} r\\t
\end{array}\right)/\left(\begin{array}{c} r\\ n(L) \end{array}\right)$ distinct subsets of size at most $t$ and at least $t-n$ whose product is 1 mod $L$.
\end{theorem}

A comprehensive proof of this theorem is given in [AGP, Proposition 1.2]; rather than reprint the proof here, we refer the interested reader to that paper.

Obviously, this theorem indicates that it will be important for us to learn about the size of $\lambda(L)$.  To this end, we note that $L$ consists of primes $q$ for which $q-1$ only has prime factors that are less than $y$.  For a prime $r$, let $a_r$ be the largest power such that $r^{a_r}\leq y^\theta$.  Then

$$\lambda(L)\leq \prod_{r\leq y,r\mbox{ }prime}r^{a_r}\leq \prod_{r\leq y,r\mbox{ }prime}y^{\theta} \leq e^{2 y\theta}.$$

Using this observation to bound (1) gives the following:


\begin{lemma} For $n(L)$ as above,
$$n(L)\leq e^{3 y\theta}.$$
\end{lemma}

\section{Counting Elements of $\mathcal P_k$}

In order to determine the size of $\mathcal P_k$, we must now invoke Conjecture 2.  To this end, we have the following:

\begin{theorem} Assume Conjecture \ref{strong}.  Then there exists an integer $k$ such that $$|\mathcal P_k|>\frac{2^\frac{\omega(L)}{A+1}}{(\log L)^A}.$$\end{theorem}\label{PKK}

We note that throughout the remainder of the paper, we will assume that $y$ and $A$ are large enough that the first prime in an arithmetic progression mod $L$ is strictly less than $L\log^A L$ (rather than $\leq$).

\begin{proof} Assume the conjecture.  Then for each $d$ that can be written as the product of $Q$'s, there exists a $k<\log^A L$ for which $dk+a$ is prime.

Now, we must prove that none of these primes are double-counted (i.e. none of the primes are being counted as $dk+a$ for two different values of $d$).  Let us assume that there exist $d_1,k_1,d_2,k_2$ such that $d_1k_1+1=d_2k_2+1$.  We know that there exists some $Q_i$ that divides $d_1$ but not $d_2$.  So $Q_i$ must divide $k_2$.  By the conjecture,
$$k_2<\log^{A} d\leq \log^{A} L.$$

But for any $Q_i$,
$$Q_i>\frac{y^{(A+1)\theta}}{\log^{A+1} y}>\log^{A+1-\epsilon} L,$$
contradicting that $Q$ can divide $k_2$.

Since the number of possible $d$'s is $2^{[\frac{\omega(L)}{A+1}]}$ and every $k$ is $<\log^A L$, there must therefore exist a $k$ where

$$|\mathcal P_k|\geq \frac{2^\frac{\omega(L)}{A+1}}{(\log L)^A}.$$

\end{proof}

\section{Invoking the Conjecture Again}

Now, we will use the Heath-Brown conjecture again to prove that there exists another prime which we can multiply by some of the primes in $\mathcal P_k$ to find an $a$-Carmichael number.

\begin{lemma} \label{BigP}
Assume Conjecture \ref{strong}.  If $y$ and $L$ are sufficiently large then for the $k$ chosen in Theorem \ref{PKK}, there exists a prime $P$ such that $P\equiv a$ (mod $Lk$) and $\frac{P-a}{Lk}\leq \log^{A} L$.
\end{lemma}

This follows from the statement of the conjecture.

Define $k'=\frac{P-a}{Lk}$.  Then we have the following:

\begin{theorem}\label{PK}
For $k$ and $k'$ as above, there exists a subset of $\mathcal P_k$ whose product is 1 mod $Lkk'$.
\end{theorem}

\begin{proof} Recall that $$|\mathcal P_k|>\frac{2^\frac{\omega(L)}{A+1}}{(\log L)^A}.$$  Since $k\leq \log^{A} L$, $k'\leq \log^{A+1} L$ and $\lambda(L)\leq e^{2y\theta}$, it is clear that

$$\lambda(kk'L)\leq \lambda(k)\lambda(k')\lambda(L)\leq \left(\log^{2A+1} L\right) e^{2y\theta},$$
which means

$$n(kk'L)\leq e^{3y\theta}.$$

Hence, $$|\mathcal P_k|>n(kk'L),$$
and thus the proof follows from Theorem \ref{main}.
\end{proof}

Let $p_1,p_2,....,p_s$ be a set of primes in $\mathcal P_k$ whose product is 1 mod $Lkk'$.  We will denote

$$n'=p_1p_2....p_s.$$

From this, we finally prove the existence of an $a$-Carmichael number:

\begin{theorem} Let $P$ and $n'$ be as above, and let $n=Pn'$.  Then $n$ is an $a$-Carmichael number.\end{theorem}

\begin{proof} We know that $n'\equiv 1$ (mod $Lkk'$).  Since $P\equiv a$ (mod $Lkk'$) by construction, we know that $Pn'\equiv a$ (mod $Lkk'$).  So for every prime $p|n'$, we have

$$p-a=dk|Lkk'|Pn'-a.$$
Moreover, for $P$, we know that $$P-a=Lkk'|Pn'-a.$$Thus, $n$ is an $a$-Carmichael number.\end{proof}

Since there exist infinitely many $L$ that can be constructed in this fashion, there are infinitely many $a$-Carmichael numbers.

\section{Counting the $a$-Carmichael numbers}
If we wish to count the number of $a$-Carmichael numbers up to $x$, we have the following:

\begin{theorem}
Assuming Conjecture \ref{strong}, there exists a constant $C$ such that the number of $a$-Carmichael numbers up to $X$ is
$$\gg X^{\frac{C}{(\log\log\log X)^2}},$$
\end{theorem}

\begin{proof} In order to apply Theorem \ref{main}, we must choose an $r$ and $t$ and find a bound for $n=n(Lkk')$ such that $n<t<r<|\mathcal P_k|$.  As such, let
\begin{gather*}
r=\left(\frac{7}{4}\right)^{\frac{\omega(L)}{A+1}},\\
t=\left(\frac 32\right)^{\frac{\omega(L)}{A+1}}.
\end{gather*}
Since $n\leq e^{3y\theta}$ and $\omega(L)\geq \gamma y^\theta/\log y$, it is clear that both $n<\left(\frac 54\right)^{\frac{\omega(L)}{A+1}}$ and $n\leq \frac{1}{20}t$ for sufficiently large $y$.  According to Theorem \ref{main}, the number of subsets of $\mathcal P_k$ consisting of at most $t$ and at least $t-n(L)$ terms whose product is 1 mod $Lkk'$ is at least

$$\left(\begin{array}{c} r \\ t \end{array}\right)/\left(\begin{array}{c} r \\ n \end{array}\right)$$
Recalling the standard bound that

\[\left(\frac{u}{v}\right)^v\leq \left(\begin{array}{c} u\\v\end{array}\right)\leq \left(\frac{ue}{v}\right)^v,\]
we have that

\begin{align*}
\left(\begin{array}{c} r \\ t \end{array}\right)/ & \left(\begin{array}{c} r \\ n \end{array}\right)\\
\geq & \left(\left(\frac{7}{6}\right)^{\frac{\omega(L)}{A+1}} \right)^t /\left(\left(\frac{7}{5}\right)^{\frac{\omega(L)}{A+1}} e \right)^{\frac{1}{20} t}\\
\geq & \left(\left(\frac{7}{6}\right)/\left(\frac 75\right)^{\frac{1}{20}}\right)^{\frac{t(\omega(L))}{A+1}} \left(\frac 1e\right)^{\frac{1}{20} t}\\
\geq & (1.1)^{\frac{t(\omega(L))}{A+1}}\\
\geq & e^{\frac{t\log 1.1}{A+1} \left(\frac{\gamma y^\theta}{\log y}\right)}
\end{align*}
where the last line comes from the fact that $\omega(L)\geq \frac{\gamma y^\theta}{\log y}$.  Define
\[X=Px^t,\]
where $P$ is the prime given in Lemma \ref{BigP}.  We recall that $L\leq e^{\kappa y^\theta}$  for some $\kappa<1.02$ (see [RS, Theorems 7 and 8]).  So
\begin{gather} \label{g1}
X=Px^t\leq (L \log^{2A+1}L)(L\log^{A+1} L)^t< e^{3y^\theta t}.
\end{gather}
Note that for sufficiently large $y$,
\begin{gather}\label{g2}
(\log\log\log X)^2\geq \log y.
\end{gather}
From (\ref{g1}) and (\ref{g2}), it follows that
$$e^{ty^\theta\left(\frac{\log 1.1}{A+1}\right) \left(\frac{\gamma }{\log y}\right)}\gg X^{\left(\frac{\log 1.1}{3(A+1)}\right) \left(\frac{\gamma}{\left(\log\log\log X\right)^2}\right)}.$$
This is then a lower bound for the number of $a$-Carmichael numbers up to $X$, thereby proving the main theorem.

\end{proof}

\end{document}